\documentclass[11pt, reqno]{amsart}
\usepackage[margin=3.6cm]{geometry}
\usepackage{array,booktabs,tabularx}
\usepackage{graphicx}
\usepackage{amssymb}
\usepackage{enumerate}
\usepackage{color}
\usepackage{esint}
\usepackage{mathtools}
\usepackage{dsfont}
\usepackage{ esint }

\newtheorem{theorem}{Theorem}

\newtheorem{Thm}[theorem]{Theorem}
\newtheorem{proposition}[theorem]{Proposition}
\newtheorem{lemma}[theorem]{Lemma}

\theoremstyle{definition}
\newtheorem{definition}{Definition}
\newtheorem{remark}{Remark}

\newcommand{\leb}{\lambda}
\newcommand{\W}{\Omega}
\newcommand{\R}{{\mathbb R}}

\newcommand{\ep}{\varepsilon}

\newcommand{\gives}{\ensuremath{\rightarrow}}
\newcommand{\x}{\ensuremath{\times}}

\newcommand{\abs}[1]{\ensuremath{\left| #1 \right|}}
\newcommand{\lr}[1]{\ensuremath{\left(#1 \right)}}
\newcommand{\norm}[1]{\left\lVert#1\right\rVert}

\newcommand{\w}{\omega}
\newcommand{\dell}{\ensuremath{\partial}}
\newcommand{\set}[1]{\ensuremath{\{#1\}}}

\def\XXint#1#2#3{{\setbox0=\hbox{$#1{#2#3}{\int}$} \vcenter{\hbox{$#2#3$}}\kern-.5\wd0}}

\DeclareMathOperator{\supp}{supp}
\DeclareMathOperator{\inj}{inj}

\title[Scaling Asymptotics]{$C^\infty$ scaling asymptotics for the spectral projector of the Laplacian}

\author[Y. Canzani]{Yaiza Canzani}
\author[B. Hanin]{Boris Hanin}

\address[Y. Canzani]{Department of Mathematics, Harvard University, Cambridge, United States.\medskip}
 \email{canzani@math.harvard.edu}
\address[B. Hanin]{Department of Mathematics, MIT, Cambridge, United States.\medskip}
\email{bhanin@mit.edu}
\begin{document}
\maketitle
\begin{abstract}
This article concerns new off-diagonal estimates on the remainder and its derivatives in the pointwise Weyl law on a compact $n$-dimensional Riemannian manifold. As an application, we prove that near any non self-focal point, the scaling limit of the spectral projector of the Laplacian onto frequency windows of constant size is a normalized Bessel function depending only on $n$. 
\end{abstract}


\section{Introduction}
Let $(M,g)$ be a compact, smooth, Riemannian manifold without boundary. We assume throughout that the  dimension of $M$ is  $n\geq 2$ and write $\Delta_g$ for the non-negative Laplace-Beltrami operator. Denote the spectrum of $\Delta_g$ by
 \[0=\leb_0^2<\leb_1^2\leq \leb_2^2\leq \cdots \nearrow \infty.\]
This article concerns the behavior of the Schwarz kernel of the projection operators
\[E_I:L^2(M) \to \bigoplus_{\leb_j \in I} \ker(\Delta_g -\leb_j^2),\]
where $I \subset [0, \infty)$. Given an orthonormal basis $\set{\varphi_j}_{j=1}^\infty$ of $L^2(M,g)$ consisting of real-valued eigenfunctions,
\begin{equation}\label{E:BasisDef}
  \Delta_g \varphi_j = \leb_j^2 \varphi_j \qquad \text{and}\qquad \norm{\varphi_j}_{L^2}=1,
\end{equation}
the Schwarz kernel of $E_I$ is
\begin{equation}
  \label{eq:1}
  E_I(x,y)=\sum_{\leb_j\in I}\varphi_j(x)\varphi_j(y).
\end{equation}
 The study of $E_{[0,\leb]}(x,y)$ as $\leb\gives \infty$ has a long history, especially when $x=y$. For instance, it has been studied notably in \cite{DG, Hor, Iv, Ivr} for its close relation to the asymptotics of the spectral counting function
\begin{equation}
\#\set{j:\, \leb_j \leq \leb}=\int_M E_{[0,\leb]}(x,x)dv_g(x),\label{E:SCF}
\end{equation}
where $dv_g$ is the Riemannian volume form. An important result, going back to H\"ormander \cite[Thm 4.4]{Hor}, is the pointwise Weyl law (see also \cite{CH, SZ}), which says that there exists $\ep>0$ so that if the Riemannian distance $d_g(x,y)$ between $x$ and $y$ is less than $\ep$, then
\begin{equation}
E_{[0,\leb]}(x,y)=\frac{1}{\lr{2\pi}^n}\int_{\abs{\xi}_{g_y}<\leb}e^{i\langle \exp_y^{-1}(x),\,\xi\rangle}\frac{d\xi}{\sqrt{\abs{g_y}}}+R(x,y,\leb).\label{E:PWL}
\end{equation}
The integral in \eqref{E:PWL} is over the cotangent fiber $T_y^*M$ and the integration measure is the quotient of the symplectic form $d\xi\wedge dy$ by the Riemannian volume form $dv_g=\sqrt{\abs{g_y}}dy.$ In H\"ormander's original theorem, the phase function $\langle \exp_y^{-1}(x),\,\xi\rangle$ is replaced by any so-called adapted phase function and one still obtains that
\begin{equation}\label{E: big O}
\sup_{d_g(x,y)<\ep}\abs{\nabla_x^j \nabla_y^k R(x,y,\leb)}=O(\leb^{n-1+j+k})
\end{equation}
as $\leb \gives \infty,$ where $\nabla$ denotes covariant differentiation. The estimate \eqref{E: big O} for $j=k=0$ is already in \cite[Thm 4.4]{Hor}, while the general case follows from the wave kernel method (e.g. as in \S 4 of \cite{Sog} see also \cite[Thm 3.1]{BinX}). 

Our main technical result, Theorem \ref{T:Main}, shows that the remainder estimate \eqref{E: big O} for $R(x,y,\leb)$ can be improved from $O(\leb^{n-1+j+k})$ to $o(\leb^{n-1+j+k})$ under the assumption that $x$ and $y$ are near a non self-focal point (defined below). This paper is a continuation of \cite{CH} where the authors proved Theorem \ref{T:Main} for $j=k=0.$ An application of our improved remainder estimates is Theorem \ref{T:Scaling Limit}, which shows that we can compute the scaling limit of $E_{(\leb, \leb+1]}(x,y)$ and its derivatives near a non self-focal point as $\leb\gives \infty.$ 

\begin{definition}\label{D:NSF}
A point $x\in M$ is \textit{non self-focal} if the loopset 
\[\mathcal L_x:=\{\xi \in S_x^*M :  \; \exists\, t>0\; \text{with}\; \exp_{x}(t\xi)=x\}\]
has measure $0$ with respect to the natural measure on $T_x^*M$ induced by $g$. Note that $\mathcal L_x$ can be dense in $S_x^*M$ while still having measure $0$ (e.g. for points on a flat torus).
\end{definition}

\begin{Thm}\label{T:Scaling Limit}
Let $(M,g)$ be a compact, smooth, Riemannian manifold of dimension $n \geq 2$, with no boundary. Suppose $x_0\in M$ is a non self-focal point and consider a non-negative function $r_\leb$ satisfying $r_\leb=o(\leb)$ as $\leb\gives \infty.$ Define the rescaled kernel
\[E_{(\leb,\leb+1]}^{x_0}\lr{u,v}:=\leb^{-(n-1)}E_{(\leb,\leb+1]}\lr{\exp_{x_0}\lr{\frac{u}{\leb}},\,\, \exp_{x_0}\lr{\frac{v}{\leb}}}.\]
Then, for all $k,j\geq 0$, 
\[\sup_{\abs{u},\abs{v}\leq r_\leb}\abs{\dell_u^j \dell_v^k \lr{E_{(\leb, \leb+1]}^{x_0}\lr{u,v}- \frac{1}{\lr{2\pi}^n}\int_{S_{x_0}^*M}e^{i\langle u -v,\w \rangle}d\w}}=o(1)\]
as $\leb \gives \infty.$ The inner product in the integral over the unit sphere $S_{x_0}^*M$ is with respect to the flat metric $g(x_0)$ and $d\w$ is the hypersurface measure on $S_{x_0}^*M$ induced by $g(x_0).$
\end{Thm}

\begin{remark}
Theorem \ref{T:Scaling Limit} holds for $\Pi_{(\leb,\leb+\delta]}$ with arbitrary fixed $\delta>0.$ The difference is that the limiting kernel is multiplied by $\delta$ and the rate of convergence in the $o(1)$ term depends on $\delta.$ 
\end{remark}
\begin{remark}
One can replace the shrinking ball $B(x_0,r_\leb)$ in Theorem \ref{T:Scaling Limit} by a compact set $S \subset M$ in which for any $x,y \in S$ the measure of the set of geodesics joining $x$ and $y$ is zero (see Remark 3 after Theorem \ref{T:Main}).
\end{remark}
In normal coordinates at $x_0,$ Theorem \ref{T:Scaling Limit} shows that the scaling limit of $E_{(\leb,\leb+1]}^{x_0}$ in the $C^\infty$ topology is 
\[E^{\R^n}_{1}(u,v)= \frac{1}{\lr{2\pi}^n}\int_{S^{n-1}}e^{i\langle u -v,\w \rangle}d\w,\]
which is the kernel of the frequency $1$ spectral projector for the flat Laplacian on $\R^n.$ Theorem \ref{T:Scaling Limit} can therefore be applied to studying the local behavior of random waves on $(M,g).$ More precisely, a frequency $\leb$ monochromatic random wave $\varphi_\leb$ on $(M,g)$ is a Gaussian random linear combination 
\[\varphi_\leb=\sum_{\leb_j\in (\leb, \leb+1]}a_j \varphi_j\qquad a_j \sim N(0,1)~~\;\text{i.i.d,}\]
of eigenfunctions with frequencies in $\leb_j\in (\leb, \leb+1].$ In this context, random waves were first introduced by Zelditch in \cite{Zel2}. Since the Gaussian field $\varphi_\leb$ is centered, its law is determined by its covariance function, which is precisely $E_{(\leb, \leb+1]}(x,y).$ In the language of Nazarov-Sodin \cite{NS} (cf \cite{CS, SW}), the estimate \eqref{E:Near-Diag Rem Est} means that frenquency $\lambda$ monochromatic random waves on $(M,g)$ have frequeny $1$ random waves on $\R^n$ as their translation invariant local limits at every non self-focal point. This point of view is taken up in the forthcoming article \cite{CH2}. 

\begin{theorem}\label{T:Main}
Let $(M,g)$ be a compact, smooth, Riemannian manifold of dimension $n \geq 2$, with no boundary. Let $K\subseteq M$ be the set of all non self-focal points in $M.$ Then for all $k,j\geq 0$ and all $\ep>0$ there is a neighborhood $\mathcal U=\mathcal U(\ep,k,j)$ of $K$ and constants $\Lambda=\Lambda(\ep,k,j)$ and $C=C(\ep,k,j)$ for which
  \begin{equation} \label{E:Near-Diag Rem Est}
 \norm{R(x,y,\leb)}_{C_x^k(\mathcal U)\x C_y^j (\mathcal U)}\leq \ep \leb^{n-1+j+k}+C\leb^{n-2+j+k}
\end{equation}
for all $\leb>\Lambda.$ Hence, if $x_0\in K$ and $\mathcal U_\leb$ is any sequence of sets containing $x_0$ with diameter tending to $0$ as $\leb \gives \infty$, then
\begin{equation}
\norm{R(x,y,\leb)}_{C_x^k(\mathcal U_\leb)\x C_y^j(\mathcal U_\leb)}=o(\leb^{n-1+j+k}).\label{E:Little oh Rem Est}
\end{equation}
\end{theorem}
\begin{remark}
One can consider more generally any compact $S\subseteq M$ such that all $x,y\in S$ are mutually non-focal, whic means
$$\mathcal L_{x,y}:=\{\xi \in S_x^*M :  \; \exists\, t>0\; \text{with}\; \exp_{x}(t\xi)=y\}$$
has measure zero. Then, combining \cite[Thm 3.3]{Saf} with Theorem \ref{T:Main}, for every $\ep>0,$ there exists a neighborhood $\mathcal U=\mathcal U(\ep,j)$ of $S$ and constants $\Lambda=\Lambda(\ep,j,S)$ and $C=C(\ep,j,S)$ such that
\[\sup_{x,y\in S} \abs{\nabla_x^j \nabla_y^j R(x,y,\leb)}\leq \ep \leb^{n-1+2j}+C\leb^{n-2+2j}.\]
We believe that this statement is true even when the number of derivatives in $x,y$ is not the same but do no take this issue up here. 
\end{remark}
Our proof of Theorem \ref{T:Main} relies heavily on the argument for Theorem 1 in \cite{CH}, which treated the case $j=k=0.$ That result was in turn was based on the work of Sogge-Zelditch \cite{SZ, SZII}, who studied $j=k=0$ and $x=y.$ This last situation was also studied (independently and significantly before \cite{CH, SZ, SZII}) by Safarov in \cite{Saf} (cf \cite{SV}) using a somewhat different method. The case $j=k=1$ and $x=y$ is essentially Proposition 2.3 in \cite{Zel2}. We refer the reader to the introduction of \cite{CH} for more background on estimates like \eqref{E:Near-Diag Rem Est}.

\section{Proof of Theorem \ref{T:Main}}
\noindent
Let $x_0$ be a non-self focal point. Let $I,J$ be multi-indices and set \[\Omega:=\abs{I}+\abs{J}.\] 
Using that $\int_{S^{n-1}}e^{i\langle u, w\rangle} dw=(2\pi)^{n/2} J_{\frac{n-2}{2}}(|u|) |u|^{-\frac{n-2}{2}}$ for all $u \in \R^n$, we have
\begin{equation}\label{E:Bessel}
\frac{1}{\lr{2\pi}^n}\int_{\abs{\xi}_{g_y}<\leb}e^{i\langle \exp_y^{-1}(x),\,\xi\rangle}\frac{d\xi}{\sqrt{\abs{g_y}}}=\int_0^\leb \frac{\mu^{n-1}}{(2\pi)^{\frac{n}{2}}}\lr{\frac{J_{\frac{n-2}{2}}\lr{\mu d_g(x,y)}}{(\mu d_g(x,y))^{\frac{n-2}{2}}}}d\mu. 
\end{equation}
Choose coordinates around $x_0$. We seek to show that there exists a constant $c>0$ so that for every $\ep>0$ there is an open neighborhood $\mathcal U_\ep$ of $x_0$ and a constant $c_\ep$ so that we have
\begin{equation}\label{E:Main Goal}
\sup_{x,y \in \mathcal U_\ep}\left| \partial_x^I\partial_y^J E_\leb(x,y)- \int_0^\leb \frac{\mu^{n-1}}{(2\pi)^{\frac{n}{2}}}\partial_x^I\partial_y^J\lr{\frac{J_{\frac{n-2}{2}}\lr{\mu d_g(x,y)}}{(\mu d_g(x,y))^{\frac{n-2}{2}}}} d\mu \right|\leq  c\,\ep \leb^{n-1+\Omega}+c_\ep \leb^{n-2+\Omega}.
\end{equation}
Let $\rho \in \mathcal S(\R)$ satisfy $\supp \lr{\hat{\rho}}\subseteq \lr{-\inj(M,g),\inj(M,g)}$ and 
\begin{equation}\label{E:Rho Def}
 \hat{\rho}(t)= 1 \qquad \text{for all}\qquad  \abs{t}<\tfrac{1}{2}{\,\inj(M,g)}.
 \end{equation}
We prove \eqref{E:Main Goal} by first showing that it holds for the convolved measure $\rho*\dell_x^I \dell_y^J E_\leb(x,y)$ and then estimating the difference
$\abs{\rho*\dell_x^I \dell_y^J E_\leb(x,y)- \dell_x^I \dell_y^J E_\leb(x,y)}$ in the following two propositions.

\begin{proposition}\label{P:Step 1}
Let $x_0$ be a non-self focal point. Let $I,J$ be multi-indices and set $\Omega=\abs{I}+\abs{J}.$ There exists a constant $c$ so that for every $\ep>0$ there exist an open neighborhood $\mathcal U_\ep$ of $x_0$ and a constant $c_\ep$ so that we have
\[\abs{\rho*\dell_x^I \dell_y^J E_\leb(x,y)-    \int_0^\leb \frac{\mu^{n-1}}{(2\pi)^{\frac{n}{2}}}\partial_x^I\partial_y^J\lr{\frac{J_{\frac{n-2}{2}}\lr{\mu d_g(x,y)}}{(\mu d_g(x,y))^{\frac{n-2}{2}}}}d\mu } \leq c\,\ep \leb^{n-1+\Omega}+c_\ep \leb^{n-2+\Omega},\]
for all $x,y\in \mathcal U_\ep$.
\end{proposition}

\begin{proposition}\label{P:Step 2}
Let $x_0$ be a non-self focal point. There exists a constant $c$ so that for every $\ep>0$ there exist an open neighborhood $\mathcal U_\ep$ of $x_0$ and a constant $c_\ep$ so that for all multi-indices $I, J$ we have
  \[\sup_{x,y\in \mathcal U_\ep}\abs{\rho*\dell_x^I \dell_y^J E_\leb(x,y)- \dell_x^I \dell_y^J E_\leb(x,y)}\leq c\,\ep \leb^{n-1+\Omega}+c_\ep \leb^{n-2+\Omega}.\]
\end{proposition}

The proof of Proposition \ref{P:Step 2} hinges on the fact that $x_0$ is a non self-focal point. Indeed, for each $\ep>0$, Lemma 15 in \cite{CH} (which is a generalization of Lemma 3.1 in \cite{SZ}) yields the existence of a neighborhood $\mathcal O_\ep$ of $x_0$, a function $\psi_\ep \in C^\infty_c(M)$ and operators $B_\ep,C_\ep\in \Psi^0(M)$ supported in $\mathcal O_\ep$ satisfying both:
\begin{align}
&\bullet  \supp(\psi_\ep) \subset \mathcal O_\ep  \;\; \text{and} \;\; \psi_\ep=1\; \text{ on a neighborhood of}\,  x_0, \label{E: O ep}\\
&\bullet  B_\ep+C_\ep=\psi_\ep^2. \label{E: B and C}
\end{align}
The operator $B_\ep$ is built so that it is microlocally supported on the set of cotangent directions that generate geodesic loops at $x_0$. Since $x_0$ is non self-focal, the construction can be carried so that the principal symbol $b_0(x,\xi)$ satisfies $\|b_0(x, \cdot)\|_{L^2(B^*_xM)} \leq \ep$ for all $x \in M$. The operator $C_\ep$ is built so that 
$U(t)C_\ep^*$ is a smoothing operator for $\;\tfrac{1}{2}{\inj(M,g)}<\abs{t}<\frac{1}{\ep}.$
In addition, the principal symbols of $B_\ep$ and $C_\ep$  are real valued and their sub-principal symbols vanish in a neighborhood of $x_0$ (when regarded as operators acting on half-densities). 

In what follows we use the construction above to decompose $E_\leb$, up to an $O(\leb^{-\infty})$ error, as
\begin{equation}\label{E: B+C}
  E_\leb(x,y)= E_\leb B_\ep^*(x,y)+E_\leb C_\ep^*(x,y)
\end{equation}
for all $x,y$ sufficiently close to $x_0$. This decomposition is valid since  $\psi_\ep\equiv 1$ near $x_0$.
\subsection{Proof of Proposition \ref{P:Step 1}} The proof of Proposition \ref{P:Step 1} consists of writing $$\rho*\dell_x^I \dell_y^J E_\leb(x,y)=\int_0^\leb \partial_\mu(\rho*\dell_x^I \dell_y^J E_\mu(x,y)) \, d\mu,$$ and on finding an estimate for $\partial_\mu(\rho*\dell_x^I \dell_y^J E_\mu(x,y))$. Such an estimate is given in Lemma \ref{L:Smoothed Estimate}, which is stated for the more general case $\partial_\mu(\rho*\dell_x^I \dell_y^J E_\mu Q^*(x,y))$ with $Q \in \{Id, B_\ep, C_\ep\}$ that is needed in the proof of Proposition \ref{P:Step 2}. 

\begin{lemma}\label{L:Smoothed Estimate}
Let $(M,g)$ be a compact, smooth, Riemannian manifold of dimension $n\geq 2$, with no boundary. Let $Q\in \set{Id,B_\ep,C_\ep}$ have principal symbol $D_0^Q.$ Consider $\rho$ as in \eqref{E:Rho Def}, and define
\[\Omega=\abs{I}+\abs{J}.\]
Then, for all $x,y \in M$ with $d_g(x,y)\leq \tfrac{1}{2}\inj(M,g),$ all multi-indices $I,J,$ and all $\mu \geq 1$, we have 
 \begin{align}
&\partial_\mu(\rho* \dell_x^I \dell_y^J E_\mu Q^*)(x,y)\notag \\
\quad &= \frac{\mu^{n-1}}{(2\pi)^n }\dell_x^I \dell_y^J\lr{\int_{S_y^*M}e^{i \mu \langle \exp_y^{-1}(x),\w \rangle_{g_y}} \lr{D_0^Q(y,\w)+\mu^{-1}D_{-1}^Q(y,\w)} \frac{d\omega}{\sqrt{|g_y|}}}\notag\\
&+ W_{I,J}(x,y,\mu). \label{E:Smoothed Est 1}
\end{align}  
Here, $d\w$ is the Euclidean surface measure on $S_y^*M,$ and $D_{-1}^Q $ is a homogeneous symbol of order  $-1$. The latter satisfy
\begin{equation}\label{E:K rln}
D_{-1}^{B_\ep}(y,\cdot)+D_{-1}^{C_\ep}(y,\cdot)=0 \qquad \forall \,y\in \mathcal O_\ep,
\end{equation}
where $\mathcal O_\ep$ is as in \eqref{E: O ep}. 
Moreover, there exists $C>0$ so that for every $\ep>0$
\begin{equation}\label{E:K rln}
\sup_{x,y\in \mathcal O_\ep}\abs{\int_{S_y^*M}e^{i \langle \exp_y^{-1}(x),\w \rangle_{g_y}} {D_{-1}^Q(y,\w)} \frac{d\omega}{\sqrt{|g_y|}}}\leq C\, \ep.
\end{equation}
Finally, $W_{I,J}$ is a smooth function in $(x,y)$ for which there exists $C>0$ such that for all $x,y$ satisfying $d_g(x,y)\leq \frac{1}{2}\inj(M,g)$ and all $\mu >0$
\begin{equation}\label{E:W long range}
|W_{I,J}(x,y,\mu)| \leq C  \mu^{n-2+\Omega}\lr{d_g(x,y)+ (1+\mu)^{-1}}.
\end{equation}
\end{lemma}

\begin{remark}
  Note that Lemma \ref{L:Smoothed Estimate} does not assume that $x,y$ are near an non self-focal point. 
\end{remark}

\begin{remark}\label{R:other Q}
We note that Lemma \ref{L:Smoothed Estimate} is valid for more general operators $Q$. Indeed, if $Q \in \Psi^k(M)$ has vanishing subprincipal symbol (when regarded as an operator acting on half-densities), then 
\eqref{E:Smoothed Est 1} holds with $D_0^Q(y, \w)$ substituted by $\mu^k D_k^Q(y, \w)$ and with $\mu^{-1}D_{-1}^Q(y,\w)$ substituted by $\mu^{k-1}D_{k-1}^Q(y,\w)$. Here, $D_k^Q$ is the principal symbol of $Q$ and  $D_{k-1}^Q$ is a homogeneous polynomial of degree $k-1$. In this setting, the error term satisfies $|W_{I,J}(x,y,\mu)| \leq C  \mu^{n+k-2+\W}\lr{d_g(x,y)+ (1+\mu)^{-1}}$.
\end{remark}
\begin{proof}[Proof of Lemma \ref{L:Smoothed Estimate}]
We use that
\begin{equation}\label{E: convolution}
\partial_\mu (\rho * EQ^*)(x,y,\leb)= \frac{1}{2\pi} \int_{-\infty}^{+\infty} e^{it \lambda} \hat \rho (t) \,U(t) Q^*(x,y) dt,
\end{equation}
where $Q\in \Psi(M)$ is any pseudo-differential operator and $U(t)=e^{-it\sqrt{\Delta_g}}$ is the half-wave propagator. The argument from here is identical to that of \cite[Proposition 12]{CH}, which relies on a parametrix for the half-wave propagator for which the kernel can be controlled to high accuracy when $x$ and $y$ are close to the diagonal. The main corrections to the proof of \cite[Proposition 12]{CH} are that $\dell_x^I\dell_y^J$ gives an $O(\mu^{n-3+\Omega})$ error in equations (54) and (60), and gives an $O(\mu^{n-1})$ error in (59). We must also take into account that $\dell_x\Theta(x,y)^{1/2}$ and $\dell_y \Theta(x,y)^{1/2}$ are both $O(d_g(x,y)).$ 
\end{proof}

\begin{proof}[Proof of Proposition \ref{P:Step 1}]
Following the technique for proving \cite[Proposition 7]{CH}, we obtain Proposition \ref{P:Step 1} by applying Lemma \ref{L:Smoothed Estimate} to $Q=Id$  (this gives $D_0^{Id}=1$ and $D_{-1}^{Id}=0$) and integrating the expression in \eqref{E:Smoothed Est 1} from $\mu=0$ to $\mu=\leb.$
One needs to choose $\mathcal U_\ep$ so that its diameter is smaller than $\ep$, since this makes $\int_0^\leb W_{I,J}(x,y,\mu) d\mu=O(\ep \leb^{n-1+\Omega}+ \leb^{n-2+\Omega})$ as needed. One also uses identity \eqref{E:Bessel} to obtain the exact statement in Proposition \ref{P:Step 1}.
\end{proof}

\subsection{Proof of Proposition \ref{P:Step 2}}
As in \eqref{E: B+C}, 
\[  E_\leb(x,y)= E_\leb B_\ep^*(x,y)+E_\leb C_\ep^*(x,y)+O\lr{\leb^{-\infty}}\]
for all $x,y$ sufficiently close to $x_0.$ Proposition \ref{P:Step 2} therefore reduces to showing that there exist a constant $c$ independent of $\ep,$ a constant $c_\ep=c_\ep(I,J,x_0)$, and a neighborhood $\mathcal U_\ep$ of $x_0$ such that
\begin{align}
\label{E:Step 2B} \sup_{x,y\in \mathcal U_\ep}\abs{ \dell_x^I \dell_y^J E_\leb B_\ep^*(x,y)-\rho*\dell_x^I \dell_y^J E_\leb B_\ep^*(x,y)}&\leq c\,\ep \leb^{n-1+\Omega}+c_\ep \leb^{n-2+\Omega},
\shortintertext{and}
\label{E:Step 2C} \sup_{x,y\in \mathcal U_\ep}\abs{ \dell_x^I \dell_y^J E_\leb C_\ep^*(x,y)-\rho*\dell_x^I \dell_y^J E_\leb C_\ep^*(x,y)}&\leq c\,\ep \leb^{n-1+\Omega}+c_\ep \leb^{n-2+\Omega}.
\end{align}
Our proofs of \eqref{E:Step 2B} and \eqref{E:Step 2C} use that these estimates hold on diagonal when $\abs{I}=\abs{J}=0$ (i.e. no derivatives are involved). This is the content of the following result, which was proved in \cite{SZ} for $Q=Id$. Its proof extends without modification to general $Q \in \Psi^0(M).$

\begin{lemma}[Theorem 1.2 and Proposition 2.2 in \cite{SZ}]\label{L:PWL}
Let $Q\in \Psi^0(M)$ have real-valued principal symbol $q$. Fix a non-self focal point $x_0\in M$ and write $\sigma_{sub}(QQ^*)$ for the subprincipal symbol of $QQ^*$ (acting on half-densities). Then, there exists $c>0$ so that for every $\ep>0$ there exist a neighborhood $\mathcal O_\ep $ and a constant $C_\ep$ making
\[QE_\leb Q^*(x,x)
= \lr{2\pi}^{-n} \int_{|\xi|_{g_x}<\leb}\lr{ |q(x,\xi)|^2+\sigma_{sub}(QQ^*)(x,\xi)}\; \frac{d\xi}{\sqrt{\abs{g_x}}} +R_Q(x,\leb),\]
with 
\[\sup_{x\in \mathcal U}\abs{R_Q(x,\leb)}\leq c\, \ep \leb^{n-1}+C_\ep \leb^{n-2}\]
for all $\leb \geq 1.$
\end{lemma}

We prove relation  \eqref{E:Step 2B} in Section \ref{S:2B} and relation \eqref{E:Step 2C} in Section \ref{S:2C}.

\subsubsection{Proof of relation  \eqref{E:Step 2B}}\label{S:2B}
Define 
\[g_{I,J}(x,y,\leb):= \dell_x^I \dell_y^J E_\leb B_\ep^*(x,y)-\rho*\dell_x^I \dell_y^J E_\leb B_\ep^*(x,y).\]
Note that  $g_{I,J}(x,y,\cdot)$ is a piecewise continuous function. We aim to find $c, c_\ep$ and $\mathcal U_\ep$ so that $x_0\in \mathcal U_\ep$ and
\begin{equation}\label{E: g}
\sup_{x,y\in \mathcal U_\ep}\abs{g_{I,J}(x,y,\leb)}\leq c\,\ep \leb^{n-1+\Omega}+c_\ep \leb^{n-2+\Omega}.
\end{equation}
By \cite[Lemma 17]{CH}, which is a Tauberian Theorem for non-monotone functions, relation \eqref{E: g} reduces to checking the following two conditions:
\begin{align}
 &\bullet \mathcal F_{\leb \to t}{( g_{I,J})}(x,y,t)=0 \;\; \text{ for all} \;\; \abs{t}<\frac{1}{2}\inj(M,g), \label{E: cond1} \\
 &\bullet  \sup_{x,y \in \mathcal U_\ep}\sup_{s\in [0,1]}\abs{g_{I,J}(x,y,\leb+s)-g_{I,J}(x,y,\leb)} \leq c\,\ep \leb^{n-1+\Omega}+c_\ep \leb^{n-2+\Omega} \label{E: cond2} .
\end{align}
By construction, $\mathcal F_{\leb \to t}{(\dell_\leb g_{I,J})}(x,y,t)=(1-\hat \rho(t))\dell_x^I \dell_y^J U(t) B_\ep^*(x,y) =0$ for all $\abs{t}<\frac{1}{2}\inj(M,g).$ Hence, since $\mathcal F_{\leb \to t}({g_{I,J}})$ is continuous at $t=0,$ we  have \eqref{E: cond1}. To prove \eqref{E: cond2} we write
\begin{align}
&\notag \sup_{s\in [0,1]}\abs{g_{I,J}(x,y,\leb+s)-g_{I,J}(x,y,\leb)} \\
&\hspace{2cm}\leq  \sup_{s\in [0,1]} \abs{\dell_x^I \dell_y^J E_{(\leb,\leb+s]}B_\ep^*(x,y)}
+\sup_{s\in [0,1]}\abs{\rho*\dell_x^I \dell_y^J E_{(\leb,\leb+s]} B_\ep^*(x,y)}.  \label{E: bound for g}
\end{align}
The second term in \eqref{E: bound for g} is bounded above by the right hand side of \eqref{E: cond2} by Lemma \ref{L:Smoothed Estimate}. To bound the first term, use Cauchy-Schwartz to get
\begin{align*}
\sup_{s\in [0,1]} \abs{ \dell_x^I \dell_y^J [E_{(\leb,\leb+s]}B_\ep^*(x,y)]}
&=\sup_{s\in [0,1]} \Big|\sum_{\leb_j\in (\leb, \leb+1]} \dell_x^I \varphi_j(x) \cdot \dell_y^J B_\ep \varphi_j(y)   \Big|\\
&\leq \sum_{\leb_j\in (\leb, \leb+1]} \abs{\left[\lr{B_\ep \dell_y^J + [\dell_y^J, B_\ep]}\varphi_j(y)\right]}\cdot \abs{\dell_x^I \varphi_j(x)}.
\end{align*}
Write $b_0$ for the principal symbol of $B_\ep.$ By construction, for all $y$ in a neighborhood of $x_0,$ we have $\partial_y b_0(y,\xi)=0$. Therefore, $\sigma_{\abs{J}-1}\lr{[\dell_y^J, B_\ep]}=i^{\abs{J}}\{\xi^J, b_0(y,\xi)\}=0$ and we conclude that $[\dell_y^J, B_\ep] \in \Psi^{\abs{J}-2}$. Thus, by the usual pointwise Weyl Law (e.g. \cite[Equation (2.31)]{SZII}), 
\begin{align*}
\sup_{s\in [0,1]} \abs{ \dell_x^I \dell_y^J [E_{(\leb,\leb+s]}B_\ep^*(x,y)]}&\leq \sum_{\leb_j\in (\leb, \leb+1]} \abs{B_\ep \dell_y^J\varphi_j(y)}\cdot \abs{\dell_x^I \varphi_j(x)}+O(\leb^{n-3+\W})
\end{align*}
Next, define for each multi-index $K \in \mathbb N^n$ the order zero pseudo-differential operator
\[P_K:= \partial^K \Delta_g^{-|K|/2}.\]
Using Cauchy-Schwarz and that $\partial^K \varphi_j = \leb_j^{\abs{K}} P_K \varphi_j,$ we find
\begin{align*}
 & \sum_{\leb_j\in (\leb, \leb+1]} \abs{B_\ep \dell_y^J\varphi_j(y)}\cdot \abs{\dell_x^I \varphi_j(x)}\\
& \qquad\qquad \leq (\leb+1)^{\Omega}[(B_\ep P_J) E_{(\leb,\leb+1]} (B_\ep P_J)^*(y,y)]^\frac{1}{2} [P_I E_{(\leb,\leb+1]} P_I^*(x,x)]^\frac{1}{2} .
\end{align*}
Again using the pointwise Weyl Law (see \cite[Equation (2.31)]{SZII}), we have $[P_I E_{(\leb,\leb+1]} P_I^*(x,x)]^\frac{1}{2}$ is $O(\leb^{\frac{n-1}{2}}).$ Next, since according to the construction of $B_\ep$ we have
\[\sup_{x \in \mathcal U_\ep}\|b_0(x, \cdot)\|_{L^2(B^*_xM)} \leq \ep\] 
and $\partial_x b_0(x,\xi)=0$ for $x$ in a neighborhood $\mathcal U_\ep$ of $x_0$,  we conclude that
\[\sup_{x \in \mathcal U_\ep}\|\sigma_{sub}(B_\ep P_J(B_\ep P_J)^*)(x, \cdot)\|_{L^2(B^*_xM)} \leq \ep^2.\] 
Proposition \ref{L:PWL} therefore shows that there exists  $c>0$ making
\begin{equation}\label{E: no commutator}
  \sup_{x,y \in \mathcal U_\ep} \abs{(B_\ep P_J) E_{(\leb,\leb+1]} (B_\ep P_J)^*(y,y)}^{\frac{1}{2}}  \leq  c {\ep} \leb^{\frac{n-1}{2}}.
\end{equation}
This proves \eqref{E: cond2}, which together with \eqref{E: cond1} allows us to conclude \eqref{E: g}.

\subsubsection{Proof of relation  \eqref{E:Step 2C}} \label{S:2C}

Write
\begin{align}\label{E: split}
\dell_x^I\dell_y^J E_\leb C_\ep^*(x,y)
&=\sum_{\leb_j\leq \leb}\leb_j^{\Omega} \lr{{P_{I}\varphi_j(x)}}\cdot \lr{C_\ep P_{J}\varphi_j(y)}+\sum_{\leb_j\leq \leb}\leb_j^{|I|} \lr{P_{I}\varphi_j(x)} \cdot \lr{[\dell^J, C_\ep]\varphi_j(y)}.
\end{align}
As before, $[\dell^J, C_\ep]\in \Psi^{\abs{J}-2}.$ Hence, by the usual pointwise Weyl law, the second term in \eqref{E: split} and its convolution with $\rho$ are both $O(\leb^{n-2+\Omega}).$ Hence,
\begin{align*}\sup_{x,y\in \mathcal U_\ep}\abs{\dell_x^I\dell_y^J E_\leb C_\ep^*(x,y)-\rho*\dell_x^I\dell_y^J E_\leb C_\ep^*(x,y)}&=\sup_{x,y\in \mathcal U_\ep}\abs{V(x,y,\leb)-\rho*V(x,y,\leb)}\\
&\hspace{3cm} +O\lr{\leb^{n+\Omega-2}},
\end{align*}
where we have set
\[V(x,y,\leb):=\dell^I E_\leb (C_\ep \dell^J)^* (x,y)=\sum_{\leb_j\leq \leb}\leb_j^{\Omega}\lr{ P_{I}\varphi_j(x)}\cdot\lr{C_\ep P_J\varphi_j(y)}.\]
Define
\begin{align}
\alpha_{I,J}(x,y,\leb)&:=V(x,y,\leb)+\frac{1}{2}\sum_{\leb_j\leq \leb}\leb_j^{\Omega}\lr{\abs{P_{I}\varphi_{\leb_j}(x)}^2+\abs{C_\ep P_{J}\varphi_{\leb_j}(y)}^2} \label{E:alpha}\\
\beta_{I,J}(x,y,\leb)&:=  \rho*V(x,y,\leb)+\frac{1}{2}\sum_{\leb_j\leq \leb}\leb_j^{\Omega}\lr{\abs{P_{I}\varphi_{\leb_j}(x)}^2+\abs{C_\ep P_{J}\varphi_{\leb_j}(y)}^2}\label{E:beta}.
\end{align}
By construction, $\alpha_{I,J}(x,y,\cdot)$ is a monotone function of $\leb$ for $x,y$ fixed, and
$\alpha_{I,J}(x,y,\leb)-\beta_{I,J}(x,y,\leb)=V(x,y,\leb)-\rho*V(x,y,\leb).$
So we aim to show that 
\begin{equation}\label{E: alpha-beta}
 \sup_{x,y\in \mathcal U_\ep}\abs{\alpha_{I,J}(x,y,\leb)-\beta_{I,J}(x,y,\leb)}\leq c\,\ep \leb^{n-1+\Omega}+c_\ep \leb^{n-2+\Omega}.
\end{equation}
We control the difference in \eqref{E: alpha-beta} applying a Tauberian theorem for monotone functions \cite[Lemma 16]{CH}. To apply it we need to show the following:
\begin{itemize}
\item  There exists $c>0$ and $c_\ep>0$ making
\begin{equation} \int_{\leb-\ep}^{\leb+\ep} \abs{\partial_\mu\beta_{I,J}(x,y,\mu)}d\mu\leq c\ep \leb^{n-1+\Omega}+c_\ep \leb^{n-2+\Omega}. \label{E: cond a}
\end{equation}
\item  For all $N$ there exists $M_{\ep,N}$ so that for all $\leb>0$
\begin{equation}
\abs{\partial_\leb \lr{\alpha_{I,J}(x,y,\cdot) -\beta_{I,J}(x,y,\cdot)}*\phi_\ep(\mu)}\leq M_{\ep,N}\lr{1+\abs{\leb}}^{-N}. \label{E: cond b}
\end{equation}
\end{itemize}
 In equation \eqref{E: cond b}  we have set $\phi_\ep(\leb):=\tfrac{1}{\ep}\phi\lr{\tfrac{\leb}{\ep}}$ for some $\phi\in \mathcal S(\R)$ chosen  so that $\supp \hat{\phi}\subseteq (-1,1)$ and $\hat{\phi}(0)=1.$

 Relation \eqref{E: cond a} follows after applying Lemma \ref{L:PWL} to the piece of the integral corresponding to the second term in \eqref{E:beta} and from applying  Lemma \ref{L:Smoothed Estimate} together with Remark \ref{R:other Q} to $\rho*V=\rho*\dell^I E_\leb Q^*$, where $Q:=C_\ep \dell^J$ has vanishing subprincipal symbol.

To verify \eqref{E: cond b} note that 
$\supp(1-\widehat{\rho})\subseteq \{t: |t|\ge \inj(M,g)/2\}$ and 
$\supp(\widehat{\phi_\ep})\subseteq \{t: |t| \leq \frac{1}{\ep}\}$.
Observe  that 
$$\partial_\leb \Big(\alpha_{I,J}(x,y,\cdot)-\beta_{I,J}(x,y,\cdot)\Big)*\phi_\ep\,(\leb)=\mathcal F_{t\gives \leb}^{-1}\Big((1-\hat{\rho}(t))\,\hat{\phi}_\ep(t)\dell^I U(t)(\dell^J C_\ep)^*(x,y)\Big)(\leb).$$
By construction $U(t)C_\ep^*$ is a smoothing operator for $\tfrac{1}{2}\inj(M,g)<\abs{t}<\frac{1}{\ep}$. Thus, so is $\dell^I U(t)\lr{\dell^J C_\ep}^*$ which implies \eqref{E: cond b}.
This concludes the proof of relation  \eqref{E:Step 2C}. \qed



\begin{thebibliography}{1}
%
%
\bibitem{Be}
P.~B\'erard. 
\newblock On the wave equation on a compact Riemannian manifold without conjugate points.
\newblock{\em Math. Z.} 155 (1977), 249-276.
%
%
%
\bibitem{BGM}
M.~Berger, P.~Gauduchon, and E.~Mazet. 
\newblock Le spectre d'une vari\'et\'e Riemannienne.
\newblock {\em Lecture Notes in Math.} Springer  (1971), 251p.

%
\bibitem{BinX}
B.~Xu.
\newblock Derivatives of the spectral function and Sobolev norms of eigenfunctions on a closed Riemannian manifold.
\newblock {\em Anal. of Global. Anal. and Geom.}, (2004) 231--252.
%

\bibitem{CH}
Y.~Canzani, B.~Hanin. 
\newblock {Scaling Limit for the Kernel of the Spectral Projector and Remainder Estimates in the Pointwise Weyl Law.}
\newblock {\em Analysis and PDE}, 8 (2015) no. 7, 1707--1731.

\bibitem{CH2}
Y.~Canzani, B.~Hanin. 
\newblock {Local Universality of Random Waves.}
\newblock {In preparation.}

\bibitem{CS}
Y.~Canzani, P.~Sarnak. 
\newblock {On the topology of the zero sets of monochromatic random waves.}
\newblock {Preprint available: arXiv:1412.4437. }

%
%
%
\bibitem{DG}
J.~Duistermaat and V.~Guillemin.
\newblock The spectrum of positive elliptic operators and periodic bicharacteristics.
\newblock {\em Inventiones Mathematicae}, 29 (1975) no. 1, 39--79.
%
%
%
%
%
%
%
\bibitem{Hor}
L.~H\"ormander.
\newblock The spectral function of an elliptic operator.
\newblock {\em Acta mathematica},  121.1 (1968), 193-218.

%
\bibitem{Iv}
 V.~Ivrii.
 \newblock Precise spectral asymptotics for elliptic operators. 
 \newblock {\em Lecture Notes in Math.} 1100, Springer, (1984).

\bibitem{Ivr}
V.~Ivrii.
\newblock{The second term of the spectral asymptotics for a Laplace Beltrami operator on
manifolds with boundary.}
\newblock {\em (Russian) Funksional. Anal. i Prolzhen.} (14) (1980), no. 2, 25--34.
%
%
%
%
%
%
%
%
%
%
\bibitem{NS}
F.~Nazarov and M.~Sodin.
\newblock  Asymptotic laws for the spatial distribution and the number of connected components of zero sets of Gaussian random functions.
\newblock {\em Preprint: http://arxiv.org/abs/1507.02017.}

\bibitem{Saf}
Yu.~Safarov.
\newblock  Asymptotics of the spectral function of a positive elliptic operator without a nontrapping condition. \newblock {\em Funktsional. Anal. i Prilozhen.} 22 (1988), no. 3, 53-65, 96
(Russian). English translation in {\em Funct. Anal. Appl.} 22 (1988), no. 3, 213-223

\bibitem{SV}
Y.~Safarov and D.~Vassiliev. 
\newblock The asymptotic distribution of eigenvalues of partial differential operators.
\newblock {\em American Mathematical Society,} 1996.
%
%
\bibitem{SW}
P.~Sarnak and I.~Wigman
\newblock Topologies of nodal sets of random band limited functions.
\newblock {\em Preprint: http://arxiv.org/abs/1510.08500 (2015).}
%
%
%
\bibitem{STZ}
C.~Sogge, J.~Toth, and S.~Zelditch. 
\newblock About the blowup of quasimodes on Riemannian manifolds.
\newblock{\em J. Geom. Anal.} 21.1 (2011), 150-173.
%
%
\bibitem{Sog}
C.~Sogge. 
\newblock Fourier integrals in classical analysis. 
\newblock {\em Cambridge Tracts in Mathematics} 105, Cambridge University Press, (1993).
%
%
%

%

\bibitem{Sog3}
C.~Sogge.
\newblock Hangzhou lectures on eigenfunctions of the Laplacian
\newblock Annals of Math Studies. 188, Princeton Univ. Press, 2014.

\bibitem{SZ}
C.~Sogge and S.~Zelditch. 
\newblock Riemannian manifolds with maximal eigenfunction growth.
\newblock {\em Duke Mathematical Journal }114.3 (2002), 387 -- 437.

\bibitem{SZII}
C.~Sogge and S.~Zelditch. 
\newblock Focal points and sup-norms of eigenfunctions.
\newblock {\em Preprint: http://arxiv.org/abs/1311.3999.}

%
%
%
%
\bibitem{Zel2}
S.~Zelditch.
\newblock Real and complex zeros of Riemannian random waves.
\newblock {\em Contemporary Mathematics} 14 (2009), 321.
%

\end{thebibliography}
\end{document}